\documentclass[oneside]{amsart}

\usepackage[letterpaper,body={13.6cm,22.5cm}, mag=1000]{geometry}
\usepackage{amssymb, amsthm, amscd, amsmath}

\begin{document}

\theoremstyle{plain}
\newtheorem{cor}{Corollary}
\newtheorem{example}{Example}
\newtheorem{lemma}{Lemma}
\newtheorem{prob}{Problem}
\newtheorem{prop}{Proposition}
\newtheorem{theorem}{Theorem}

\newcommand{\Inn}{\operatorname{Inn} }
\newcommand{\Aut}{\operatorname{Aut} }
\newcommand{\Out}{\operatorname{Out} }
\newcommand{\Z}{\operatorname{Z} }
\newcommand{\annd}{\quad \text{ and } \quad}
\newtheorem{problem}{Problem}
\newcommand{\C}{\operatorname{C} }

\title[CLASS PRESERVING AUTOMORPHISMS]{CLASS PRESERVING AUTOMORPHISMS OF  UNITRIANGULAR GROUPS}

\author{VALERIY BARDAKOV}
\address{Sobolev Institute of Mathematics, pr. ak. Koptyuga 4, Novosibirsk, 630090, Russia}
\email{bardakov@math.nsc.ru}

\author{Andrei VESNIN}
\address{Sobolev Institute of Mathematics, pr. ak. Koptyuga 4, Novosibirsk, 630090, Russia}
\email{vesnin@math.nsc.ru}

\author{MANOJ K.  YADAV}
\address{School of Mathematics, Harish-Chandra Research Institute,\\ Chhatnag Road, Jhunsi, Allahabad - 211 019, India}
\email{myadav@hri.res.in}

\begin{abstract}
Let $\textrm{UT}_n (K)$ be a unitriangular group over a field $K$
and $\Gamma_{n,k} := \textrm{UT}_n (K)/$ $ \gamma_k(\textrm{UT}_n
(K))$, where $\gamma_k \left( \mathrm{UT}_n(K)\right)$ denotes the
$k$-th term of the lower central series of $\mathrm{UT}_n (K)$, $2
\le k \le n$. We prove that the group of all class preserving
automorphisms of $\Gamma_{n,k}$ is equal to $\Inn(\Gamma_{n,k})$ if
and only if $K$ is a prime field. Let $G_n^{(m)} := \mathrm{UT}_n
(\mathbb{F}_{p^m}) / $ $\gamma_3 \left(
\mathrm{UT}_n(\mathbb{F}_{p^m}) \right)$. We calculate the group of
all class preserving automorphisms and class preserving outer
automorphisms of $G_n^{(m)}$.
\end{abstract}

\maketitle

\section{Introduction}

Let $G$ be an arbitrary group. An automorphism $\alpha$ of $G$ is called \emph{class preserving} if $\alpha(x) \in x^G$ for all $x \in G$, where $x^G$ denotes the conjugacy class of $x$ in $G$. The set of all class preserving automorphisms of $G$, which we denote here by $\Aut_c(G)$, is a normal subgroup of $\Aut(G)$, the group of all automorphisms of $G$. Notice that $\Inn(G)$, the group of all inner automorphisms of $G$, is a normal subgroup of $\Aut_c(G)$. An automorphism $\beta$ of $G$ is called \emph{normal} if $\beta(N) \subseteq N$ for all normal subgroups $N$ of $G$. The set of all normal automorphisms of $G$, which we denote here by $\Aut_n(G)$, is a normal subgroup of $\Aut(G)$. Since every normal subgroup of a group $G$ can be written as a union of certain conjugacy classes in the group $G$, it follows that every class preserving automorphism of $G$ is a normal automorphism of $G$. Thus we get the following sequence of normal subgroups of $\Aut(G)$
$$\Inn(G) \unlhd \Aut_c(G) \unlhd \Aut_n(G) \unlhd \Aut(G).$$

The group of normal automorphisms for a given group $G$ is a well studied object. See \cite{RB, En1, F-G, J,  JR,  Lub, Lue, Lul, Ne-2,  Ro, R, Rom}. We'll come back to this later in the last section.

On the other hand $\Aut_c(G)$, the group of all class preserving automorphisms of a given group $G$, is a very less studied object. It seems that W. Burnside is the first person who talked about class preserving automorphisms and posed the following question in  1911 \cite[p. 463]{B}: \emph{Does there exist any finite group $G$ such that $G$ has  a non-inner class preserving automorphism?} In 1913, Burnside himself gave an affirmative answer to this question \cite{B1}. He constructed a group $G$ of order $p^6$, $p$ an odd prime,  isomorphic to the group
$$
\mathrm{UT}_3(\mathbb{F}_{p^2})  = \left\{
\left(
\begin{array}{ccc}
  1 & 0 & 0 \\
x & 1 & 0 \\
z & y & 1 \\
\end{array}
\right)~\left| \right.~x, y, z \in \mathbb{F}_{p^2}
\right\},
$$
where $\mathbb{F}_{p^2}$ is the field consisting of $p^2$ elements. He proved that $\Aut_c(G)$ is an elementary abelian $p$-group of order $p^8$, but $|\Inn(G)| = p^4$.

In 1947 G.E.~Wall \cite{gW47} constructed examples of finite groups $G$ such that \linebreak $\Inn(G) < \Aut_c(G)$. In 1966, C.H.~Sah \cite{cS68} studied the factor group \linebreak $\Out_{c}(G) := \Aut_c(G) / \Inn(G)$ for an arbitrary finite groups $G$, using cohomological techniques. After this, $\Aut_c(G)$ was studied by M.~Hertweck \cite{mH01} in 2001 for a special class of groups $G$ and by the third author \cite{mY06} in 2007 for finite $p$-groups $G$. In 2005, G.~Endimioni \cite{En2} proved that $\Out_{c}(G) = 1$ for all free nilpotent groups $G$.

From this point onwards we'll mainly talk about unitriangular groups. Let $\mathrm{UT}_n(K)$ denote the unitriangular group consisting of all $n \times n$ unitriangular matrices having entries in a field $K$, where $n \ge 2$. Let $\Gamma_{n,k}$ denote the factor group $\mathrm{UT}_n(K)/\gamma_k(\mathrm{UT}_n(K)$, where $1 \le k \le n$. Obviously $\Gamma_{2,k}$, $\Gamma_{n,1}$ and $\Gamma_{n,2}$ are abelian and therefore all class preserving automorphsms of these groups are inner.

Since the first example of a group $G$ such that $\Inn(G) < \Aut_c(G)$ comes from unitriagular groups, it is quite natural to pose the following problems:

\medskip

\noindent {\bf Problem A.} Find necessary and sufficient conditions on $\Gamma_{n,k}$, $n \geq 3$, $3 \leq k \leq n$,  such that $\Aut_c(\Gamma_{n,k}) = \Inn(\Gamma_{n,k})$.

\medskip

\noindent {\bf Problem B.} Let $G = \mathrm{UT}_n(K)$ be such that $\Inn(G) < \Aut_c(G)$. Study the stucture of $\Aut_c(G)$.

\medskip

In this paper we study these problems.  We provide a solution to Problem A in the following theorem.

\medskip

\noindent {\bf Theorem A.}
Let $\Gamma_{n,k} = \mathrm{UT}_n(K)/\gamma_k(\mathrm{UT}_n(K))$,  where $K$ is a field,  $n \ge 3$ and $3 \le k \le n$. Then  $\Aut_c(\Gamma_{n,k}) = \Inn(\Gamma_{n,k})$ if and only if $K$ is a prime field.

\medskip

The statement follows from Theorems \ref{theorem-new} and
\ref{theorem2.7}, which we prove in Section $2$.  Alternative proof
of Theorem \ref{theorem-new}  based on the ideas suggested by
N. S.~Romanovskii is given in Appendix A.

On the lines of the proof of Theorem \ref{theorem-new}, it follows that $\Aut_c (\mathrm{UT}_n(\mathbb{Z})) = $ $\Inn (\mathrm{UT}_n (\mathbb{Z}))$ for every positive integer $n$, where $\mathbb{Z}$ denotes the ring of integers.

Next we study Problem B for the following special case. For integers $n \ge 3$, $m \ge 1$ and a prime $p$, let $\mathrm{UT}_n(\mathbb{F}_{p^m})$ denote the group of all $n \times n$ unitriangular matrices with entries in the field $\mathbb{F}_{p^m}$. Set $G_n^{(m)} := \mathrm{UT}_n(\mathbb{F}_{p^m})/\gamma_3(\mathrm{UT}_n(\mathbb{F}_{p^m}))$, where $\gamma_3 \left( \mathrm{UT}_n(\mathbb{F}_{p^m})\right)$ denotes the third term of the lower central series of $\mathrm{UT}_n(\mathbb{F}_{p^m})$. In the following theorem, which we prove in Section $3$ as Theorem 3.2, we calculate the group $\Aut_c(G_n^{(m)})$ and notice (surprisingly) that  $|\Aut_c(G_n^{(m)})/\Inn(G_n^{(m)})|$ is independent of $n$.

\medskip

\noindent {\bf Theorem B.}
Let $G_n^{(m)}$ be the group defined in the preceding paragraph. Then $\Aut_c(G_n^{(m)})$ is elementary abelian $p$-group of order $p^{2m^2+mn-3m}$. Moreover the order of $\Out_c(G_n^{(m)})$ is $p^{2m(m-1)}$, which is independent of $n$.

\medskip

We would like to remark that the  group of all automorphisms of $\mathrm{UT}_n(K)$ was studied by Levchuk~\cite{Lev}.

In Section 4 we give a quick survey on normal automorphisms of groups (defined above) and pose several interesting problems in the sequel.

Before concluding this section we give some definitions, which we'll use many times in the paper. Let $G$ be a group with a fixed set of generators $\{x_1, \cdots, x_d\}$ (say). An automorphism $\varphi$ is called \emph{basis conjugating} if $\varphi(x_i) \in x_i^G$, for all $i$ such that $1 \le i \le d$. Let $\mathrm{Cb}(G)$ denote the group of all basis conjugating  automorphisms of $G$. An automorphism of a group $G$ is said to be \emph{central} if it induces identity on $G/\Z(G)$, where $\Z(G)$ denotes the center of $G$.  We write the subgroups in the lower central series of $G$ as $\gamma_n(G)$, where $n$ runs over all strictly positive integers. They are defined inductively by
$$\gamma_1(G) = G \annd \gamma_{n+1}(G) = [\gamma_n(G), G] $$
for any integer $n \ge 1$.

\section{Unitriangular groups over fields}

In this section we consider the group $\mathrm{UT}_n(K)$ of lower-unitriangular matrices over a field or ring $K$. It is well-known that this group is generated by elementary transvections $t_{ik}(\lambda)$, where $1 \leq k < i \leq n$ and $\lambda \in K \setminus \{ 0 \}$, which satisfy the following relations (see \cite[\S~3]{KM}):
$$
[t_{ik}(\alpha), t_{mj}(\beta)]  = \left\{
\begin{array}{ll}
t_{ij}(\alpha \beta), & ~\mathrm{if}~ k = m,  \\
t_{mk}(-\alpha \beta), & ~\mathrm{if}~ i=j, \\
{\mathbf 1_n}, & ~\mathrm{if}~ k \not= m~\mathrm{and}~ i \not= j,  \\
\end{array}
\right.
$$
where $2 \le i, m \le n$ and $ 1 \le k, j \le n-1$.   Here $[a, b] = a^{-1} b^{-1} a b$ and ${\mathbf 1_n}$ is the unity matrix of degree $n$.

Notice that the group $\mathrm{UT}_n (K)$ is generated by $t_{i+1, i}(K)$ for $1 \le i \le n-1$, where  $t_{i+1, i}(K)$ is the subgroup of $\mathrm{UT}_n (K)$ generated by $\{t_{i+1, i}(\alpha) \mid \alpha \in K\}$. Thus
$$
\mathrm{UT}_n (K) = \langle t_{2,1} (K), t_{3,2}(K), \ldots , t_{n,n-1} (K) \rangle .
$$

Starting with the case $K = \mathbb Q$, we observe that the additive group $\langle \mathbb Q; + \rangle$ is infinitely generated. It is generated by the elements of the form $\frac{1}{p}$, where $p$ runs over all prime integers. So, the subgroup $t_{i+1,i}(\mathbb Q)$ of $\mathrm{UT}_n (\mathbb Q)$, where $1 \le i \le n-1$, is infinitely generated by the elements $t_{i+1, i} (\frac{1}{p})$, where $p$ runs over all prime integers.  But, in the following lemmas, we show that if $\varphi$ is a basis conjugating automorphism of $\mathrm{UT}_n (\mathbb Q)/ \gamma_k(\mathrm{UT}_n(\mathbb Q))$, then the image of $t_{i+1,i} (\mathbb Q)$ (modulo $\gamma_k(\mathrm{UT}_n(\mathbb Q))$) under $\varphi$ is determined by the image of the element $t_{i+1,i}(a/b)$ for some arbitrarily chosen fixed rational number $a / b$.

Here and below we use the same notations for transvections in $\mathrm{UT}_n(K)$ and for their images in $\mathrm{UT}_n(K)/\gamma_k(\mathrm{UT}_n(K))$, where $2 \le k \le n$.

\begin{lemma} \label{lemma_2.1}
Let $\varphi$  be a basis conjugating automorphism of  $\mathrm{UT}_n (\mathbb Q) / \gamma_{3} (\mathrm{UT}_n (\mathbb  Q))$. Then  for any $i$ the action of $\varphi$ on the subgroup $t_{i+1,i} (\mathbb Q)$ (modulo $\gamma_{3} (\mathrm{UT}_n (\mathbb  Q))$) is  uniquely determined by the image of $t_{i+1, i} (\frac{a}{b})$ for an arbitrarily chosen non-zero rational number $a / b \in \mathbb Q$.
\end{lemma}

\begin{proof}
Let us fix a non-zero rational number $a/b \in \mathbb Q$ and suppose that
$$
\varphi (t_{i+1, i} ( a / b) ) = t_{i+1, i} (a / b) \, t_{i+1, i-1} (\lambda) \, t_{i+2, i} (\mu)
$$
for some $\lambda, \mu \in \mathbb Q$, where $1 \le i \le n-1$, and we set $t_{i+1, i-1} (\lambda)= 1$ when $i = 1$ and  $t_{i+2, i} (\mu) = 1$ when $i = n-1$.  To complete the proof of the lemma, we need to find $\varphi(t_{i+1, i} (r / s))$ for any rational number $r / s$. This simply means that we need to find  $x, y \in \mathbb Q$ (depending only on $a$, $b$, $\lambda$, $\mu$, $r$ and $s$) such that
$$
\varphi (t_{i+1, i} (r / s) ) = t_{i+1,i} (r/s) \, t_{i+1, i-1} (x) \, t_{i+2, i} (y).
$$
The obvious relation
$(t_{i+1, i} (r / s))^{s a} = (t_{i+1,i} (a/b))^{b r}$, gives us
$$
\varphi ( (t_{i+1, i} (r/s) )^{s a} ) = \varphi ( ( t_{i+1, i} (a/b))^{b r} ),
$$
which is equivalent to
$$
t_{i+1, i} (ar) \, t_{i+1, i-1} (a s x) \, t_{i+2, i} (a s y) = t_{i+1, i} (ar) \, t_{i+1, i-1} (b r \lambda) \, t_{i+2,i} (b r  \mu).
$$
Thus comparing the corresponding entries of the matrices on both sides, we get
$$
a s x = b r \lambda, \qquad a s y = b r \mu.
$$
This gives the required values
$$
x = \frac{b r \lambda}{a s}, \qquad y = \frac{b r \mu}{a s},
$$
which completes the proof.
\end{proof}

In the following lemma we consider basis conjugating automorphisms $\varphi$ of the group $\mathrm{UT}_n (\mathbb Q) / \gamma_{k+1} (\mathrm{UT}_n (\mathbb  Q))$ such that for each $i$, $\varphi(t_{i+1,i}(w_i)) = (t_{i+1,i}(w_i))^{g_i}$, where  $g_i = t_{k,1}(x_{i,1})\, t_{k+1,2}(x_{i,2})\cdots t_{n,n-k+1}(x_{i,n-k+1})$. The proof of this lemma goes on the lines of the proof of Lemma \ref{lemma_2.1}.

\begin{lemma}\label{lemma_2.2}
 Let $\varphi$  be an automorphism of $\mathrm{UT}_n (\mathbb Q) / \gamma_{k+1} (\mathrm{UT}_n (\mathbb  Q))$ such that for $k \le n-k$,
$$
\varphi(t_{i+1, i}(w_i))  = \left\{%
\begin{array}{l}
t_{i+1,i} (w_i)\, t_{i+k,i}(-x_{i,i+1} w_i), \hfill\;\;  \mbox{if}~ i < k;\\
t_{i+1,i}(w_i) \,t_{i+1, i-k+1}(x_{i,i-k+1} w_i)\, t_{i+k,i}(-x_{i,i+1} w_i),  \hfill \mbox{if}~k \le i \le n-k;\\
t_{i+1,i}(w_i)\, t_{i+1, i-k+1}(x_{i,i-k+1} w_i), \hfill \mbox{if}~ i > n-k,
\end{array}%
\right.
$$
and for $k > n-k$,
$$
\varphi(t_{i+1, i}(w_i))  = \left\{%
\begin{array}{l}
t_{i+1,i} (w_i)\, t_{i+k,i}(-x_{i,i+1} w_i), \hfill \;\;\mbox{if}~ i \le  n-k;\\
t_{i+1,i}(w_i),  \hfill \mbox{if}~ n-k < i < k;\\
t_{i+1,i}(w_i)\, t_{i+1,i-k+1}(x_{i,i-k+1} w_i), \hfill \mbox{if}~ i \ge  k,
\end{array}%
\right.
$$
where $w_i, x_{i,i+1}, x_{i,i-k+1} \in \mathbb Q$. Then for any $i$, the action of $\varphi$ on the subgroup $t_{i+1,i} (\mathbb Q)$ (modulo $\gamma_{k+1} (\mathrm{UT}_n (\mathbb  Q))$) is  uniquely determined by the image of $t_{i+1, i} (\frac{a}{b})$ for an arbitrarily chosen non-zero rational number $a / b \in \mathbb Q$.
\end{lemma}
In the following theorem, we prove the sufficient part of Theorem A.

\begin{theorem} \label{theorem-new}
Let  $\Gamma_{n, k} = \mathrm{UT}_n(K) / \gamma_k \left(\mathrm{UT}_n(K)\right)$, where  $K$ is a prime field, $n\geq 3$, $2 \leq k \leq n$. Then $\Aut_c(\Gamma_{n, k}) = \Inn(\Gamma_{n, k})$.
\end{theorem}

\begin{proof}
Let us fix a positive integer $n \ge 3$. Since $K$ is a prime field, then $K = \mathbb F_p$ or $K = \mathbb Q$. The nilpotency class of $\Gamma_{n, k}$ is $(k-1)$ and   $\Z(\Gamma_{n, k}) \cong K^{n-(k-1)}$ is generated by the subgroups
$$
t_{k,1}(K), t_{k+1,2}(K), \ldots, t_{n,n-k+1}(K).
$$
It is easy to see that $\Gamma_{n, k} / \Z(\Gamma_{n, k}) \cong \Gamma_{n, k-1}$.

We use induction on $k$ to prove  $\Aut_c(\Gamma_{n, k}) = \Inn (\Gamma_{n, k})$, where  $2 \le k \le n$. If $k=2$, then $\Gamma_{n, 2} \cong K^{n-1}$ is abelian and the statement holds trivially. Before going to the inductive step, we compute the case when $k = 3$ just to show  how things work. Let $\varphi \in \Aut_c(\Gamma_{n, 3})$.  Then it follows from Lemma \ref{lemma_2.1} that $\varphi$ is completely determined by its images on $t_{i+1,i} (s)$, where $s = 1$ if $K = \mathbb F_p$ and $s = 1/ m$, for some fixed prime $m$, if $K = \mathbb Q$. Let $g = t_{2,1}(\alpha_1) \, t_{3,2} (\alpha_2) \cdots t_{n,n-1} (\alpha_{n-1})$, where  $\alpha_1, \alpha_2, \ldots , \alpha_{n-1} \in K$. Since the nilpotency class of $\Gamma_{n, 3}$ is $2$, for a fixed $s$, we have
\begin{eqnarray*}
(t_{2,1} (s))^g & = & t_{2,1}(s) \, t_{3,1}(-\alpha_2 s), \cr
(t_{3,2}(s))^g & = & t_{3,2}(s) \, t_{3,1}(\alpha_1 s) \, t_{4,2}(-\alpha_3 s), \cr & \vdots & \cr
(t_{i+1,i}(s))^g & = & t_{i+1,i} (s) \,  t_{i+1,i-1}(\alpha_{i-1} s) \, t_{i+2,i}(-\alpha_{i+1} s), \cr
& \vdots & \cr
(t_{n-1,n-2} (s))^g & = & t_{n-1,n-2} (s) \,  t_{n-1,n-3}(\alpha_{n-3} s) \, t_{n,n-2}(-\alpha_{n-1} s),  \cr
(t_{n, n-1}(s))^g & = & t_{n,n-1} (s) \,  t_{n,n-2}(\alpha_{n-2} s).
\end{eqnarray*}

Since $\varphi \in \Aut_c(\Gamma_{n, 3})$ and the nilpotency class of $\Gamma_{n, 3}$ is $2$, there exist $x_2, x_3, \ldots , $ $x_{n-1},$ $y_1, y_2, \ldots, y_{n-2} \in K$ such that
\begin{eqnarray*}
\varphi(t_{2,1}(s)) & = & t_{2,1}(s) \,  t_{3,1}(-y_1 s), \cr
\varphi(t_{3,2}(s)) & = & t_{3,2}(s) \, t_{3,1}(x_2 s) \, t_{4,2}(-y_2 s), \cr
& \vdots & \cr
\varphi(t_{i+1,i} (s)) & = & t_{i+1,i} (s) \,  t_{i+1,i-1}(x_i s) \, t_{i+2,i}(-y_i s), \cr &
\vdots & \cr \varphi(t_{n-1,n-2} (s)) & = & t_{n-1,n-2} (s) \,
t_{n-1,n-3}(x_{n-2} s) \, t_{n,n-2}(-y_{n-2} s), \cr
\varphi(t_{n,n-1}(s)) & = & t_{n,n-1}(s) \,  t_{n,n-2}(x_{n-1} s).
\end{eqnarray*}

Automorphism $\varphi$ preserves all conjugacy classes in $\Gamma_{n, 3}$, so, in particular, it also preserve the classes of elements $h_i = t_{i+1,i}(s) \, t_{i+3,i+2,}(s)$, where $i = 1, \ldots, n-3$. Acting by $\varphi$ we have
$$
\varphi(h_1) =  \varphi (t_{2,1}(s) \, t_{4,3} (s))   =  t_{2,1}(s) \, t_{4,3}(s) \, t_{3,1}(-y_1 s) t_{4,2}(x_{3} s) \, t_{5, 3} (-y_{3} s).
$$
Since $h_1$ and $\varphi(h_1)$ are in the same conjugacy class, $\varphi (h_1) = h_1^{g_1}$ for some element $g_1 \in \Gamma_{n, 3}$, say
$g_1 = t_{2,1}(\alpha_1^1) \, t_{3,2} (\alpha_2^1) \cdots t_{n,n-1} (\alpha_{n-1}^1)$. Now
\begin{eqnarray*}
h_1^{g_1}  & = & (t_{2,1}(s) \, t_{4,3}(s))^{g_1} \cr
& =  & t_{2,1} (s) \, t_{3,1}(-\alpha_{2}^1 s) \, t_{4,3} (s) \,  t_{4,2}(\alpha_{2}^1 s) \, t_{5,3}(-\alpha_{4}^1 s) \cr
& =  & t_{2,1} (s) \, t_{4,3} (s) \, t_{3,1}(-\alpha_{2}^1 s) \, t_{4,2}(\alpha_{2}^1 s) \, t_{5,3}(-\alpha_{4}^1 s).
\end{eqnarray*}
{}From $\varphi(h_1) = h_1^{g_1}$, by comparing the corresponding entries of the matrices we get
\begin{equation}
y_1 = x_3. \label{new-1}
\end{equation}
Moreover, $\alpha_2^1 = y_1$, $\alpha_4^1 = y_3$, and for $j\not\in \{ 2, 4\}$, numbers $\alpha_j^1$ can be taken arbitrarily.

For $2 \leqslant i \leqslant n-4$ we have
\begin{eqnarray*}
\varphi(h_i) & =  & \varphi (t_{i+1,i}(s) \, t_{i+3,i+2} (s))  \cr
& = & t_{i+1,i}(s) \, \, t_{i+1,i-1}(x_i s) \, t_{i+2,i}(-y_i s)  \, t_{i+3,i+2}(s) \,  t_{i+3,i+1}(x_{i+2} s) \, t_{i+4, i+2} (-y_{i+2} s) \cr
& = & t_{i+1,i}(s) \, t_{i+3,i+2}(s) \, t_{i+1,i-1}(x_i s) \,t_{i+2,i}(-y_i s) \, t_{i+3,i+1}(x_{i+2} s) \, t_{i+4, i+2} (-y_{i+2} s).
\end{eqnarray*}
Since $h_i$ and $\varphi(h_i)$ are in the same conjugacy class, $\varphi (h_i) = h_i^{g_i}$ for some element $g_i \in \Gamma_{n, 3}$, say $g_i = t_{2,1}(\alpha_1^i) \, t_{3,2} (\alpha_2^i) \cdots t_{n,n-1} (\alpha_{n-1}^i)$. Now
\begin{eqnarray*}
h_i^{g_i}  & = & (t_{i+1,i} (s) \, t_{i+3,i+2} (s))^{g_i} \cr
& = & t_{i+1,i} (s) \,  t_{i+1, i-1} (\alpha_{i-1}^i s)\, t_{i+2,i}(-\alpha_{i+1}^i s) \, t_{i+3,i+2} (s) \, t_{i+3,i+1}(\alpha_{i+1}^i s) \, t_{i+4,i+2}(-\alpha_{i+3}^i s) \cr
& = & t_{i+1,i} (s) \, t_{i+3,i+2} (s) \, t_{i+1, i-1} (\alpha_{i-1}^i s) \, t_{i+2,i}(-\alpha_{i+1}^i s)  \, t_{i+3,i+1}(\alpha_{i+1}^i s) \, t_{i+4,i+2}(-\alpha_{i+3}^i s).
\end{eqnarray*}
{}From $\varphi (h_i) = h_i^{g_i}$, we get
\begin{equation}
y_i = x_{i+2} \label{new-2}
\end{equation}
for $i=2, \ldots, n-4$. Moreover, $\alpha_{i+1}^i = y_i$, $\alpha_{i+3}^i = y_{i+2}$, and for $j\not\in \{i-1,  i+1, i+3\}$, numbers $\alpha_j^i$ can be chosen arbitrarily.

Finally, we have
\begin{eqnarray*}
\varphi(h_{n-3}) & =  & \varphi (t_{n-2,n-3}(s) \, t_{n,n-1} (s))  \cr
& = & t_{n-2,n-3}(s) \, t_{n-2, n-4}(x_{n-3} s) \, t_{n-1,n-3}(-y_{n-3} s)  \, t_{n,n-1} (s) \, t_{n,n-2}(x_{n-1} s) \cr
& = & t_{n-2,n-3}(s) \, t_{n,n-1} (s) \, t_{n-2, n-4}(x_{n-3} s) \, t_{n-1,n-3}(-y_{n-3} s)  \, t_{n,n-2}(x_{n-1} s).
\end{eqnarray*}
We have $\varphi(h_{n-3}) = h_{n-3}^{g_{n-3}}$ for some element $g_{n-3} \in \Gamma_{n, 3}$, say
$g_{n-3} = t_{2,1}(\alpha_1^{n-3})$ $t_{3,2} (\alpha_2^{n-3})$ $ \cdots t_{n,n-1} (\alpha_{n-1}^{n-3})$. So
\begin{eqnarray*}
h_{n-3}^{g_{n-3}}  & = & (t_{n-2,n-3} (s) \, t_{n,n-1} (s) )^{g_{n-3}}  \cr
& = & t_{n-2,n-3} (s) \, t_{n-2, n-4}(\alpha_{n-4}^{n-3} s) \, t_{n-1,n-3}(-\alpha_{n-2}^{n-3} s)  \, t_{n,n-1} (s) \, t_{n,n-2}(\alpha_{n-2}^{n-3} s) \cr
& = & t_{n-2,n-3} (s) \, t_{n,n-1} (s) \, t_{n-2, n-4} (\alpha_{n-4}^{n-3} s) \, t_{n-1,n-3}(-\alpha_{n-2}^{n-3} s) \, t_{n,n-2}(\alpha_{n-2}^{n-3} s).
\end{eqnarray*}
Therefore, by comparing the corresponding entries of $\varphi(h_{n-3})$ and $(h_{n-3})^{g_{n-3}}$, we get
\begin{equation}
y_{n-3} = x_{n-1}. \label{new-3}
\end{equation}
Moreover, $\alpha_{n-4}^{n-3} = x_{n-3}$, hence by (\ref{new-2}) $\alpha_{n-4}^{n-3} = y_{n-5}$, and   $\alpha_{n-2}^{n-3} = y_{n-3}$. For $j\not\in \{ n-4, n-2\}$ numbers $\alpha_j^{n-3}$ can be taken arbitrarily.

Thus, it follows from (\ref{new-1}), (\ref{new-2}) and (\ref{new-3}) that by defining
$$
\alpha_1 = x_2, \quad \alpha_2 = y_1, \quad \alpha_3 = y_2, \quad \ldots , \quad  \alpha_{n-1} = y_{n-2}
$$
and
$g = t_{2,1}(\alpha_1) \, t_{3,2} (\alpha_2) \cdots t_{n,n-1} (\alpha_{n-1})$, we obtain
$$
\varphi (t_{i+1, i} (s)) = (t_{i+1, i} (s))^g, \qquad i=1, \ldots, n-1.
$$
This proves that $\varphi$ is the inner automorphism determined by $g$. Hence $\Aut_c(\Gamma_{n, 3}) = \Inn (\Gamma_{n, 3})$

\medskip

Let us now assume  that $\Aut_c(\Gamma_{n, k}) = \Inn(\Gamma_{n, k})$ for some $k \geq 3$. The natural homomorphism  $\Gamma_{n, k+1} \longrightarrow \Gamma_{n, k}$ induces the homomorphism
$$
\Aut(\Gamma_{n, k+1})  \longrightarrow \Aut(\Gamma_{n, k}).
$$
Let $\overline{\varphi}$ denote the image of $\varphi \in \Aut(\Gamma_{n, k+1})$ under this homomorphism.

Suppose that $\varphi \in \Aut_c(\Gamma_{n, k+1})$. Then $\overline{\varphi} \in \Aut_c(\Gamma_{n, k})$. By the induction hypothesis there exists an inner automorphism $\overline{\omega} \in \Inn(\Gamma_{n, k})$ such that $\overline{\varphi}\: \overline{\omega}$ is the identity automorphism of $\Gamma_{n, k}$. Let $\omega \in \Inn(\Gamma_{n, k+1})$ be the pre-image of $\overline{\omega}$. Then $\psi := \varphi\: \omega$ is a central automorphism of $\Gamma_{n, k+1}$. Since $\psi$ is central as well as class preserving, it acts on every $t_{i+1,i}(\alpha)$, $\alpha \in K$ by conjugating it by some element $x_i$ of the form
$$
t_{k,1}(x_{i,1}) \, t_{k+1,2}(x_{i,2}) \, \cdots \, t_{n,n-k+1}(x_{i,n-k+1}),
$$
where $i = 1, 2, \ldots, n-1$ and $x_{i,j} \in K$.  By Lemma \ref{lemma_2.2}, it is sufficient to study $\psi$ on the generating elements $t_{i+1, i}(s)$, where $s = 1$ if $K = \mathbb F_p$ and $s = 1/ m$, for some fixed prime $m$, if $K = \mathbb Q$. Notice that $\psi$ is defined on these generating elements in the following way:

For $k \le n-k$,
$$
\psi(t_{i+1, i}(s))  = \left\{%
\begin{array}{l}
t_{i+1,i} (s)\, t_{i+k,i}(-x_{i,i+1} s), \hfill\;\; \mbox{if}~ i < k;\\
t_{i+1,i}(s) \,t_{i+1, i-k+1}(x_{i, i-k+1} s) \,t_{i+k,i}(-x_{i, i+1} s), \; \hfill \mbox{if}~k \le i \le n-k;\\
t_{i+1,i}(s)\, t_{i+1, i-k+1}(x_{i,i-k+1} s), \hfill \mbox{if}~ i > n-k,
\end{array}%
\right.
$$
and for $k > n-k$,
$$
\psi(t_{i+1, i}(s))  = \left\{%
\begin{array}{l}
t_{i+1,i} (s) \,t_{i+k,i}(-x_{i,i+1} s), \hfill \;\;\mbox{if}~ i \le  n-k;\\
t_{i+1,i}(s),  \hfill \mbox{if}~ n-k < i < k;\\
t_{i+1,i}(s) \,t_{i+1,i-k+1}(x_{i,i-k+1} s), \hfill \mbox{if}~ i \ge  k.
\end{array}%
\right.
$$

Also notice that an inner automorphism $f_g$ determined by
$$
g = t_{k,1}(\alpha_{1}) \, t_{k+1,2}(\alpha_{2}) \, \cdots \, t_{n,n-k+1}(\alpha_{n-k+1}), \;\; \alpha_{j} \in K \;\;\text{for}\;\; 1 \le j \le n-k+1,
$$
is defined on the generating elements in the following way:

For $k \le n-k$,
$$
f_g(t_{i+1, i}(s))  = \left\{%
\begin{array}{l}
t_{i+1,i} (s)\, t_{i+k,i}(-\alpha_{i+1} s), \hfill\;\; \mbox{if}~ i < k;\\
t_{i+1,i}(s)\, t_{i+1, i-k+1}(\alpha_{i-k+1} s)\, t_{i+k,i}(-\alpha_{i+1} s),  \hfill \mbox{if}~k \le i \le n-k;\\
t_{i+1,i}(s) \,t_{i+1, i-k+1}(\alpha_{i-k+1} s), \hfill \mbox{if}~ i > n-k,
\end{array}%
\right.
$$
and for $k > n-k$,
$$
f_g(t_{i+1, i}(s))  = \left\{%
\begin{array}{l}
t_{i+1,i} (s)\, t_{i+k,i}(-\alpha_{i+1} s), \hfill \;\;\mbox{if}~ i \le  n-k;\\
t_{i+1,i}(s),  \hfill \mbox{if}~ n-k < i < k;\\
t_{i+1,i}(s)\, t_{i+1,i-k+1}(\alpha_{i-k+1} s), \hfill \mbox{if}~ i \ge  k.
\end{array}%
\right.
$$

Let us consider the elements
$$
h_{i,k}  =  t_{i+1,i}(s)  \, t_{i+k+1,i+k} (s),
$$
where $i = 1, 2, \ldots, n-k-1$.  Then for $k \le n-k$
$$
\psi(h_{i, k}(s))  = \left\{%
\begin{array}{l}
t_{i+1,i}(s)\,t_{i+k+1,i+k} (s)\, t_{i+k,i}(-x_{i,i+1} s)\, t_{i+k+1, i+1}(x_{i+k, i+1} s)\\ \;\; \cdot t_{i+2k,i+k}(-x_{i+k, i+k+1} s), \hfill\;\; \mbox{if}~ i < k~\mbox{and}~i+k \le n-k ;\\
t_{i+1,i}(s)\,t_{i+k+1,i+k} (s)\, t_{i+k,i}(-x_{i,i+1} s)\,t_{i+k+1, i+1}(x_{i+k, i+1} s),\\ \hfill\;\; \mbox{if}~i < k~\mbox{and}~i+k > n-k;\\
t_{i+1,i}(s)\, t_{i+k+1,i+k} (s)\, t_{i+1, i-k+1}(x_{i, i-k+1} s)\, t_{i+k,i}(-x_{i, i+1} s)\\ \;\; \cdot t_{i+k+1, i+1}(x_{i+k, i+1} s) \,t_{i+2k,i+k}(-x_{i+k, i+k+1} s),\\  \hfill \mbox{if}~k \le i \le n-k~\mbox{and}~k \le i+k \le n-k ;\\
t_{i+1,i}(s)\, t_{i+k+1,i+k} (s)\, t_{i+1, i-k+1}(x_{i, i-k+1} s) \,t_{i+k,i}(-x_{i, i+1} s)\\ \;\; \cdot t_{i+k+1, i+1}(x_{i+k,i+1} s),  \hfill \mbox{if}~k \le i \le n-k~\mbox{and}~ i+k > n-k,
\end{array}%
\right.
$$
and for $k > n-k$,
$$
\psi(h_{i, k}(s)) = t_{i+1,i} (s)\, t_{i+k+1,i+k} (s)\, t_{i+k,i}(-x_{i,i+1} s)\,t_{i+k+1,i+1}(x_{i+k,i+1} s),
$$
for all $i \le  n-k-1$.

The reason why we are considering $h_{i,k} = t_{i+1,i}(s)  \, t_{i+k+1,i+k} (s)$ is that $\psi(t_{i+1,i})$ and $\psi(t_{j+1,j})$ both involve conjugation by elements from the same subgroup  ($t_{i+k,i+1}(K)$) only when $|i - j| = k$. More precisely, when $i+k =j$,  $\psi(t_{i+1,i})$ and $\psi(t_{j+1,j})$ involve $(t_{i+1,i})^{t_{i+k,i+1}(x_{i, i+1})}$ and $(t_{j+1,j})^{t_{j,i+1}(x_{j, i+1})}$ respectively.  If  $|i - j| \not= k$, then $t_{i+1,i}(s)$ and $t_{i+1,i}(s)$ centralize the subgroups $t_{j+k,j+1}(K)$ and $t_{i+k,i+1}(K)$ respectively.

So to prove $\psi$ is inner, it is sufficient to prove that $x_{i, i+1} = x_{i+k, i+1}$ for $1 \le i \le n-k-1$. Since $\psi$ is class preserving, there exists an element 
$$g_i = t_{k,1}(y_{i,1}) \, t_{k+1,2}(y_{i,2}) \, \cdots \, t_{n,n-k+1}(y_{i,n-k+1})$$
 such that $\psi(h_{i,k}) = h_{i,k}^{g_i}$ for $1 \le i \le n-k-1$. Now we calculate $h_{i,k}^{g_i}$.

For $k \le n-k$
$$
h_{i,k}^{g_i}  = \left\{%
\begin{array}{l}
t_{i+1,i} (s)\,t_{i+k+1,i+k} (s)\, t_{i+k,i}(-y_{i,i+1} s) \,t_{i+k+1, i+1}(y_{i, i+1} s)\\ \;\; \cdot t_{i+2k,i+k}(-y_{i+k, i+k+1} s), \hfill\;\; \mbox{if}~ i < k~\mbox{and}~i+k \le n-k ;\\
t_{i+1,i} (s)\,t_{i+k+1,i+k} (s) \,t_{i+k,i}(-y_{i,i+1} s)\,t_{i+k+1, i+1}(y_{i, i+1} s),\\ \hfill\;\; \mbox{if}~i < k~\mbox{and}~i+k > n-k;\\
t_{i+1,i}(s) \,t_{i+k+1,i+k} (s)\, t_{i+1, i-k+1}(y_{i, i-k+1} s) \,t_{i+k,i}(-y_{i, i+1} s)\\ \;\; \cdot t_{i+k+1, i+1}(y_{i, i+1} s) \,t_{i+2k,i+k}(-y_{i+k, i+k+1} s),\\  \hfill \mbox{if}~k \le i \le n-k~\mbox{and}~k \le i+k \le n-k ;\\
t_{i+1,i}(s) \,t_{i+k+1,i+k} (s) \,t_{i+1, i-k+1}(y_{i, i-k+1} s) \,t_{i+k,i}(-y_{i, i+1} s)\\ \;\; \cdot t_{i+k+1, i+1}(y_{i,i+1} s),  \hfill \mbox{if}~k \le i \le n-k~\mbox{and}~ i+k > n-k,
\end{array}%
\right.
$$
and for $k > n-k$,
$$
h_{i,k}^{g_i} = t_{i+1,i} (s)\, t_{i+k+1,i+k} (s) \,t_{i+k,i}(-y_{i,i+1} s) \, t_{i+k+1,i+1}(y_{i,i+1} s),
$$
for all $i \le  n-k-1$.

Since $\psi(h_{i,k}) = h_{i,k}^{g_i}$, by comparing the corresponding entries of the matrices we get
$$
x_{i, i+1} = y_{i, i+1}, \quad x_{i+k, i+1} = y_{i, i+1}.
$$
Hence $x_{i, i+1} = x_{i+k, i+1}$, which proves that $\psi$ is an inner automorphism of $\Gamma_{n, k+1}$, induced by $g = t_{k,1}(\alpha_{1}) \, t_{k+1,2}(\alpha_{2}) \, \cdots \, t_{n,n-k+1}(\alpha_{n-k+1})$,  where $\alpha_1 = x_{k,1},\; \alpha_2 = x_{1,2},\; \ldots,\;\alpha_i = x_{i-1, i},\; \ldots, \; \alpha_{n-k+1} = x_{n-k,n-k+1}$. We would like to remark here that conjugation by $t_{k,1}(x_{k,1})$ appears only in $\psi(t_{k+1, k}(s))$ and  conjugation by $t_{n, n-k+1}(x_{n-k,n-k+1})$ appears only in $\psi(t_{n-k-1, n-k}(s))$. That is why we have taken  $\alpha_1 = x_{k,1}$ and  $\alpha_{n-k+1} = x_{n-k,n-k+1}$. Since $\psi = \varphi\:\omega$ with $\omega \in \Inn(\Gamma_{n, k+1})$, it follows that $\varphi$ is an inner automorphism of $ \Gamma_{n, k+1}$. As $\varphi$ was an arbitrary class preserving automorphism of $ \Gamma_{n, k+1}$, we have  $\Aut_c(\Gamma_{n, k+1}) = \Inn(\Gamma_{n, k+1})$ for all $k \ge 3$.
\end{proof}

The proof of the following theorem goes on the lines of the proof of Theorem \ref{theorem-new}.
\begin{theorem}
$\Aut_c (\mathrm{UT}_n(\mathbb{Z})) = \Inn(\mathrm{UT}_n (\mathbb{Z}))$, where $\mathbb{Z}$ is the ring of integers and $n$ is any positive integer.
\end{theorem}

Now we prove the necessary part of Theorem A.

\begin{theorem}\label{theorem2.7}
Let  $\Gamma_{n, k} = \mathrm{UT}_n(K) / \gamma_k \left(\mathrm{UT}_n(K)\right)$, where  $K$ is not a prime field, $n\geq 3$, $3 \leq k \leq n$. Then $\Aut_c(\Gamma_{n, k}) \neq \Inn(\Gamma_{n, k})$.
\end{theorem}

\begin{proof}
 It is sufficient to construct a non-inner class preserving automorphism of $\Gamma_{n, k}$. Consider $K$ as a vector space over $K_0$, where $K_0$ is its prime subfield, with the basis $\{e^0, e^1, e^2, \ldots \}$, where $e^0 = 1$. Then the group  $t_{i,j} (K)$ is generated by elements $t_{i,j}(e^l)$, $l = 0, 1 ,\ldots$, if $K_0 = \mathbb F_p$ and by elements $t_{i,j}(\frac{1}{q} e^l)$, $l=0,1,\ldots$ and $q$ runs over all primes, if $K_0 = \mathbb Q$.

By Lemma \ref{lemma_2.2}, to define an automorphism $\varphi$ of the group $\mathrm{UT}_n(K)$ of the form given below, it is enough to define its images $\varphi(t_{i+1,i}(e^0))$  and $\varphi(t_{i+1,i}(\frac{1}{q} e^l))$ for $l =  1, 2, \ldots$. Let us consider the following map $\psi$ of $\Gamma_{n, k}$ into itself:
\begin{eqnarray*}
\psi(t_{k,k-1} (e^0)) & = & t_{k,k-1} (e^0) \, t_{k,1} (e^0 c), \cr
\psi(t_{k,k-1} (\frac{1}{q} e^j)) & = & t_{k,k-1} (\frac{1}{q} e^j) ~\mbox{for}~~ j \ge 1, \cr
\psi(t_{i+1,i} (e^0)) & = & t_{i+1,i} (e^0) ~\mbox{for}~~i \neq k-1,\cr
\psi(t_{i+1,i} (\frac{1}{q} e^l)) & = & t_{i+1,i} (\frac{1}{q} e^l) ~\mbox{for}~~i \neq k-1 ~\mbox{and}~ l=1, 2,\ldots,
\end{eqnarray*}
where $c$ is a  non-zero elements of $K$ and $q$ runs over all primes. It is easy to see that $\psi$ is a central automorphism of $\Gamma_{n, k}$.

First we show that $\psi$ is class preserving. Any element $g \in G$ can be represented in the form
$$
g = t_{2,1} (\alpha_1) \cdots t_{k, k-1} (\alpha_{k-1}) \cdots t_{n,n-1} (\alpha_{n-1}) \, d,
$$
where $d \in \gamma_2(\Gamma_{n, k})$. If $\alpha_{k-1} = 0$, then $\psi$ fixes $g$.  So assume that $\alpha_{k-1} \neq 0$. Since the automorphism $\psi$ is central, it fixes every element of $\gamma_2(\Gamma_{n, k})$. Thus by the  definition of $\psi$, we get
$$
\psi (g) = g \, t_{k,1} (\beta)
$$
for some $\beta \in K$, since $t_{k, 1}(\beta) \in \Z(\Gamma_{n, k})$.  Let $h = t_{k-1,1} ( \beta / \alpha_{k-1})$. Then
\begin{eqnarray*}
(t_{k,k-1}(\alpha_{k-1}))^h & = & t_{k-1,1} (- \beta/\alpha_{k-1}) \, t_{k,k-1}(\alpha_{k-1}) \, t_{k-1,1}
( \beta / \alpha_{k-1}) \\ & = & t_{k,k-1}(\alpha_{k-1}) \,  t_{k,1} (\beta)
\end{eqnarray*}
and for any $u \in K$,
$$(t_{r,s}(u))^h = t_{r,s}(u)~\mbox{if}~~ (r,s) \neq (k,k-1).$$
Therefore, $\psi(g) = g^h$, which shows that $\psi$ is class preserving.

Now we show that $\psi$ is not inner. Let us assume the contrary, and therefore there must exist some  $h \in \Gamma_{n, k}$ such that
\begin{eqnarray}
(t_{k,k-1} (e^0))^h & = & t_{k,k-1} (e^0) \, t_{k,1} (e^0 c),  \label{eq10}\\
(t_{k,k-1} (\frac{1}{q} e^j))^h & = & t_{k,k-1} (\frac{1}{q} e^j) ~\mbox{for}~~ j \ge 1, \label{eq11}\\
(t_{i+1,i} (\frac{1}{q} e^l))^h & = & t_{i+1,i} (\frac{1}{q} e^l) ~\mbox{for}~~i \neq k-1 ~\mbox{and}~ l=0,1,\ldots, \label{eq12}
\end{eqnarray}
where $q$ runs over all primes. Taking into account \eqref{eq10} and \eqref{eq12}, notice that the only possibility for the value of $h$ is $t_{k-1,1}(c)$ modulo $\Z(\Gamma_{n, k})$. Now by inserting the value of $h$ in \eqref{eq11}, we get
$$
(t_{k,k-1} (\frac{1}{q} e^j)) t_{k,1}(\frac{1}{q} e^j c)  =  t_{k,k-1} (\frac{1}{q} e^j) ~\mbox{for}~~ j \ge 1 ~\mbox{ and each prime}~~q.
$$
Since $e^j \neq 0$, it is possible only when $c = 0$, which contradicts our choice of $c$ being non-zero.  Hence $\psi$ is not inner. This completes the proof of the theorem.
\end{proof}

\begin{cor}
Let $\Gamma_{n, k} = \mathrm{UT}_n(\mathbb{Q}[x])/\gamma_k(\mathrm{UT}_n(\mathbb{Q}[x]))$. Then  $\Aut_c(\Gamma_{n, k}) \neq \Inn (\Gamma_{n, k})$, where  $n\geq 3$, $3 \leq k \leq n$.
\end{cor}

\begin{proof}
We can consider  $\mathbb{Q}[x]$ as an infinite dimensional vector space over $\mathbb{Q}$ with basis $e^0 = 1, e^1 = x, e^2 = x^2, \ldots$. Thus the proof follows on the lines of the proof of  Theorem \ref{theorem2.7}.
\end{proof}

We remark that the case $n=3$ of Theorem A was proved by another techniques in \cite{Garge}.

\section{Quotient groups of unitriangular groups \\ over finite fields}

Throughout this section $K$ always denotes the field $\mathbb{F}_{p^m}$, where $m$ is a positive integer. Let us denote by $G_n^{(m)}$, $n \geqslant 3$, $m \geqslant 1$, the quotient group:
$$
\label{eqn2}
G_n^{(m)} = \mathrm{UT}_n(\mathbb{F}_{p^m}) / \gamma_3 \left( \mathrm{UT}_n(\mathbb{F}_{p^m}) \right).
$$
Obviously $G_n^{(m)}$ is a nilpotent group of class $2$.

The group $G_3^{(1)}$ is known as the \emph{Heisenberg group modulo $p$}. It is a group of order $p^3$ with two generators $x, y$ (say) and relations
$$
z = x^{-1} y^{-1} x y , \qquad x^p = y^p = z^p = 1, \qquad x z = z x, \qquad y z = z y .
$$
It is easy to check that $|G_n^{(1)}| = p^{2n-3}$ and $|Z(G_n^{(1)})| = p^{n-2}$, so
$$
|\Inn(G_n^{(1)})| = |G_n^{(1)} / Z(G_n^{(1)})| = p^{n-1} \quad  \mathrm{and} \quad |\mathrm{Cb}(G_n^{(1)})| = p^{(n-1)(n-2)}.
$$

We remark that $G_3^{(2)}$ is the group constructed by W. Burnside, which is mentioned in the introduction.

Let us represent the elements of the field $\mathbb{F}_{p^m}$ in the following form:
$$
\mathbb{F}_{p^m} = \{ a_0 \cdot 1 + a_1 \cdot \theta + a_2 \cdot
\theta^2 + \cdots + a_{n-1} \cdot \theta^{m-1} ~|~ a_i  \in
\mathbb{F}_{p} \},
$$
where $\theta$ is a root of the minimal irreducible polynomial over~$\mathbb{F}_p$. Then $G_n^{(m)}$ is generated by the transvections
$$
t_{i+1,i}(\theta^k), \quad i = 1, \ldots, n-1, \quad k=0, \ldots ,
m-1
$$
of order $p$. The commutator of  $t_{\alpha, \beta} (\theta^k)$ and $t_{\gamma, \delta} (\theta^{\ell})$ for $\alpha > \beta$, $\gamma > \delta$ is defined as follows:
$$
[t_{\alpha, \beta} (\theta^k), t_{\gamma, \delta} (\theta^{\ell})] = \left\{%
\begin{array}{lc}
t_{\alpha, \delta}(\theta^{k+\ell}), & \textrm{if} \, \beta = \gamma
, \\
t_{\gamma, \beta} (- \theta^{k + \ell}), & \textrm{if} \, \alpha  = \delta ,\\ {\mathbf 1_n}, & \textrm{otherwise} .
\end{array}%
\right.
$$

Here and below we use the same notations for transvections in $\mathrm{UT}_n(\mathbb{F}_{p^m})$ and for their images in $G_n^{(m)}$.

A group $G$ is said to be \emph{Camina group} if $x^G = x\gamma_2(G)$ for all $x \in G \setminus \gamma_2(G)$. The following lemma can also be derived as a corollary of \cite[Theorem 5.4]{mY06} by noticing that $G_3^{(m)}$ is a Camina special $p$-group.  But we here give a direct simple proof.

\begin{lemma} \label{lemma_n=3}
Any basis conjugating automorphism of $G_3^{(m)}$, $m \geqslant 1$, is class preserving, i.e., $\mathrm{Cb}(G_3^{(m)}) = \Aut_c(G_3^{(m)})$.
\end{lemma}

\begin{proof}
The group $G_3^{(m)}$ is generated by $2m$ elements $t_{2,1} (\theta^k)$ and $t_{3,2} (\theta^k)$, where $k=0, \ldots, m-1$.
Notice that any automorphism $\varphi \in \textrm{Cb}( G_3^{(m)})$ acts on the generators as follows:
\begin{eqnarray*}
\varphi (t_{2,1} (\theta^k)) & = & t_{2,1} (\theta^k) \, t_{3,1} (-y^{(k)}),  \cr \varphi (t_{3,2} (\theta^k)) & = & t_{3,2} (\theta^k) \, t_{3,1} (x^{(k)})
\end{eqnarray*}
for some $x^{(k)}, y^{(k)} \in \mathbb F_{p^m}$, where $k=0, \ldots, m-1$. It is easy to see that $\varphi$ is a central automorphism of $G_3^{(m)}$. The automorphism $\varphi$ is class preserving if and only if for each $h \in G_3^{(m)}$ there exists an element $g$ such that $\varphi(h) = h^{g}$. Let $h = t_{2,1} (\lambda_1) \, t_{3,2} (\lambda_2) \, t_{3,1} (\mu_1)$ be an arbitrary element of $G_3^{(m)}$. Since $\varphi$ is a central automorphism, we have
\begin{eqnarray*}
\varphi(t_{2,1} (\lambda_1)) & = & t_{2,1} (\lambda_1) \, t_{3,1} (- P), \cr \varphi(t_{3,2}) (\lambda_2) & = & t_{3,2} (\lambda_2) \, t_{3,1} (Q)
\end{eqnarray*}
for some $P, Q \in \mathbb F_{p^m}$. Thus
$$
\varphi(h) = t_{2,1} (\lambda_1) \, t_{3,2} (\lambda_2) \, t_{3,1} (-P + Q + \mu_1).
$$
Suppose that there exists $g \in G_3^{(m)}$ such that $\varphi(h) = h^{g}$. The element $g$ can be represented in its normal form as
$g = t_{2,1} (\alpha_1) \, t_{3,2} (\alpha_2) \, t_{3,1} (\beta_1)$ for some $\alpha_1, \,\alpha_2,\,\beta_1 \in \mathbb{F}_{p^m}$.
Consider the conjugation of $g$ on the generators of $G_3^{(m)}$, which is given by
\begin{eqnarray*}
(t_{2,1} (\theta^k))^{g} & = & t_{2,1} (\theta^k) \, t_{3,1} (-\alpha_2 \theta^k),  \cr (t_{3,2} (\theta^k))^{g} & = & t_{3,2} (\theta^k) \, t_{3,1} (\alpha_1 \theta^k).
\end{eqnarray*}
Therefore
$$
h^{g} = t_{2,1} (\lambda_1) \, t_{3,2} (\lambda_2) \, t_{3,1} (\alpha_1 \lambda_2 - \alpha_2 \lambda_1 + \mu_1).
$$
Thus $\varphi(h) = h^{g}$ if and only if
$$
\alpha_1 \lambda_2 - \alpha_2 \lambda_1 = - P + Q.
$$
But for any given $P$, $Q$ (depending on $\varphi$) and for any given $\lambda_1$, $\lambda_2$ (depending on $h$), one can always find $\alpha_1$ and $\alpha_2$ such that $\alpha_1 \lambda_2 - \alpha_2 \lambda_1 = - P + Q$. Hence any basis conjugating automorphism of $G_3^{(m)}$ is class preserving, which completes the proof of the lemma.
\end{proof}

Now we prove Theorem B.

\begin{theorem}[Theorem B] \label{theorem_3.1}
Let $G_n^{(m)}$ be the group defined above. Then $\Aut_c (G_n^{(m)})$ is elementary abelian $p$-group of order $p^{2m^2+mn-3m}$. Moreover, the order of  $\Out_c(G_n^{(m)})$ is $p^{2m(m-1)}$, which is independent of $n$.
\end{theorem}

\begin{proof}
As mentioned above, $G_n^{(m)}$ is generated by $\{t_{i+1,i}(\theta^k) \mid 1 \le i \le n-1, \;0 \le k \le m-1\}$.
Observe that
\begin{eqnarray*}
(t_{\beta+1, \beta}(y))^{t_{\alpha+1, \alpha} (x)} & = &
t_{\beta+1, \beta}(y) \, [t_{\beta+1, \beta}(y), t_{\alpha+1, \alpha}(x)] \\
& = &
\left\{%
\begin{array}{lc}
t_{\beta+1, \beta}(y) \, t_{\beta+1, \beta-1} (yx) , & \text{if} \,
\beta= \alpha + 1, \\
t_{\beta+1, \beta}(y) \, t_{\beta+2, \beta} (-yx) , & \text{if} \,
\beta = \alpha - 1, \\ t_{\beta+1, \beta} (y), & \text{otherwise}.
\end{array}\nonumber
\right.
\end{eqnarray*}

Let $g \in G_n^{(m)}$ be an arbitrary element. Then it can be written in the normal form:
$$
\begin{gathered}
g = t_{2,1}(\alpha_1) t_{3,2}(\alpha_2) \cdots
t_{n,n-1}(\alpha_{n-1}) t_{3,1}(\beta_1) t_{4,2}(\beta_2) \cdots
t_{n,n-2}(\beta_{n-2}).
\end{gathered}
$$
Since $G_n^{(m)}$ is nilpotent of class 2, for any $j=1, \ldots,
n-2$, the element $t_{j+2,j}(\beta_j)$ lies in the center of $G_n^{(m)}$. So while considering the conjugation action by $g$, we can assume, without loss of any generality, that
$$
g = t_{2,1}(\alpha_1) \, t_{3,2}(\alpha_2) \, \cdots \,
t_{n,n-1}(\alpha_{n-1}).
$$
Thus $g$ acts on the generators of $G_n^{(m)}$ as follows:
\begin{eqnarray*}
(t_{2,1}(\theta^k))^g & = & t_{2,1} (\theta^k) \, t_{3,1}(-\alpha_2
\theta^k), \\
(t_{3,2}(\theta^k))^g & = & t_{3,2} (\theta^k) \, t_{3,1}(\alpha_1
\theta^k) \, t_{4,2}(-\alpha_3 \theta^k), \\
&\vdots&\\
(t_{i+1,i}(\theta^k))^g & = & t_{i+1,i} (\theta^k)  \,
t_{i+1,i-1}(\alpha_{i-1} \theta^k) \, t_{i+2,i}(-\alpha_{i+1}
\theta^k), \\
&\vdots&\\
(t_{n-1,n-2}(\theta^k))^g & = & t_{n-1,n-2} (\theta^k)  \,
t_{n-1,n-3}(\alpha_{n-3} \theta^k) \, t_{n,n-2}(-\alpha_{n-2}
\theta^k), \\ (t_{n,n-1}(\theta^k))^g & = & t_{n,n-1} (\theta^k) \,
t_{n,n-2}(\alpha_{n-2} \theta^k) ,
\end{eqnarray*}
where $k=0, \ldots, m-1$ and $\alpha_i \in \mathbb F_{p^m}$ for
$i=1, \ldots, n-1$.

Let $\varphi$ be a basis conjugating automorphism $G_n^{(m)}$. Then $\varphi$ can be defined on the generating elements of $G_n^{(m)}$ as follows:
\begin{eqnarray*}
\varphi (t_{2,1}(\theta^k)) & = & t_{2,1} (\theta^k) \,
t_{3,1}(-y_1^{(k)}), \\
\varphi (t_{3,2}(\theta^k)) & = & t_{3,2} (\theta^k) \,
t_{3,1}(x_2^{(k)}) \, t_{4,2}(-y_2^{(k)}), \\
&\vdots&\\
\varphi (t_{i+1,i}(\theta^k)) & = & t_{i+1,i} (\theta^k) \,
t_{i+1,i-1}(x_i^{(k)}) \, t_{i+2,i}(-y_i^{(k)}), \\
&\vdots&\\
\varphi (t_{n-1,n-2}(\theta^k)) & = & t_{n-1,n-2} \, (\theta^k)
t_{n-1,n-3}(x_{n-2}^{(k)}) \, t_{n,n-2}(-y_{n-2}^{(k)}), \\ \varphi
(t_{n,n-1}(\theta^k)) & = & t_{n,n-1} (\theta^k) \,
t_{n,n-2}(x_{n-1}^{(k)}),
\end{eqnarray*}
for some elements $x_i^{(k)}$ and $y_j^{(k)}$ of $\mathbb F_{p^m}$, where  $i=2,\ldots,n-1$, $j=1,\ldots,n-2$, and $k=0,\ldots,m-1$.

This shows that
$$
|\Aut_c(G_n^{(m)})| \le |\mathrm{Cb}(G_n^{(m)})| \le p^{m^2(2n-4)},
$$
since each $x_i$ as well as $y_j$ can take at most $p^m$ values in $\mathbb F_{p^m}$. We are now going to show that if $\varphi$ is a class preserving automorphism of  $G_n^{(m)}$, then many of $x_i^{(k)}$ and $y_j^{(k)}$ depend on one another. So assume that $\varphi$ is a class preserving automorphism of  $G_n^{(m)}$. Let us consider the elements
$$
h_i^{p,q} = t_{i+1,i}(\theta^p) \, t_{i+3,i+2} (\theta^q),
$$
where $p,q \in \{0,1,\ldots, m-1\}$ and $i = 1, \ldots, n-3$.

Then
$$
\varphi(h_1^{p,q})  =  t_{2, 1} (\theta^p) \, t_{4, 3} (\theta^q)  \,  t_{3,1}  (-y_1^{(p)}) \, t_{4,2} (x_3^{(q)}) \, t_{5,3} (-y_3^{(q)}), $$
\begin{eqnarray*}
\varphi(h_i^{p,q}) & =  & t_{i+1,i} (\theta^p)  \, t_{i+3,i+2} (\theta^q)  \,  t_{i+1, i-1} (x_i^{(p)})  \\ & &  t_{i+2, i} (-y_i^{(p)})   \, t_{i+3,i+1} (x_{i+2}^{(q)}) \, t_{i+4,i+2} (-y_{i+2}^{(q)})
\end{eqnarray*}
for $1< i < n-3$, and
$$
\varphi(h_{n-3}^{p,q})  =  t_{n-2, n-3} (\theta^p) \, t_{n,n-1} (\theta^q)  \,  t_{n-2,n-4}  (x_{n-3}^{(p)}) \, t_{n-1,n-3} (-y_{n-3}^{(p)}) \, t_{n,n-2} (x_{n-1}^{(q)}).
$$

Since $\varphi$ is class preserving, for each $h_i^{p,q}$ there exits $g_i^{p,q} \in G_n^{(m)}$ such that $\varphi(h_i^{p,q}) = (h_i^{p,q})^{g_i^{p,q}}$. Let $g_i^{p,q} = t_{2,1}(\alpha_{i,1}^{p,q}) \, t_{3,2}(\alpha_{i,2}^{p,q}) \, \cdots \, t_{n,n-1}(\alpha_{i,n-1}^{p,q})$, where $1 \le i \le n-3$ and $0 \le p,\, q \le m-1$.

The conjugation action of $g_i^{p,q}$ on $h_i^{p,q}$ is given by
\begin{eqnarray*}
(h_1^{p,q})^{g_1^{p,q}} & = & t_{2, 1}(\theta^p)  t_{4, 3} (\theta^q) \, t_{3, 1} (-\alpha_{1,2}^{p,q}\, \theta^p)  \, t_{4,2} (\alpha_{1,2}^{p,q} \, \theta^q) \, t_{5, 3} (-\alpha_{1,4}^{p,q} \, \theta^q),
\end{eqnarray*}
\begin{eqnarray*}
(h_i^{p,q})^{g_i^{p,q}} & = & t_{i+1, i}(\theta^p)  \, t_{i+3, i+2}  \,(\theta^q)  \, t_{i+1, i-1}(\alpha_{i,i-1}^{p,q} \,\theta^p) \cr & & \,\cdot  t_{i+2, i} (-\alpha_{i,1+1}^{p,q} \, \theta^p) \,  t_{i+3,i+1} (\alpha_{i,1+1}^{p,q} \, \theta^q) \, t_{i+4, i+2} (-\alpha_{i,i+3}^{p,q} \, \theta^q)
\end{eqnarray*}
for $1< i < n-3$, and
\begin{eqnarray*}
(h_{n-3}^{p,q})^{g_{n-3}^{p,q}} & = & t_{n-2, n-3}(\theta^p)  \, t_{n, n-1} (\theta^q)  \, t_{n-2, n-4}  (\alpha_{n-3,n-4}^{p,q}  \,\theta^p) \cr & & \,\cdot t_{n-1, n-3} (- \alpha_{n-3,n-2}^{p,q} \, \theta^p) \, t_{n, n-2} (\alpha_{n-3,n-2}^{p,q}  \,\theta^q). \cr
\end{eqnarray*}

Since $\varphi(h_i^{p,q}) = (h_i^{p,q})^{g_i^{p,q}}$, by comparing the corresponding entries of the matrices we get
\begin{equation}
y_i^{(p)} \theta^{-p} = x_{i+2}^{(q)}\, \theta^{-q} , \label{eqnxy}
\end{equation}
where $p,q \in \{ 0, 1, \ldots , m-1 \}$ and $i=1, \ldots, n-3$. We remark that this set of equations doesn't contain parameters $x_2^{(k)}$ and $y_{n-2}^{(k)}$ for $k=0, \ldots, m-1$. Putting  $q=0$ in (\ref{eqnxy})  we get
\begin{equation}\label{eq1}
y_i^{(p)} = x_{i+2}^{(0)} \, \theta^p,
\end{equation}
where $p=0, \ldots, m-1$ and $i=1, \ldots, n-3$. Now putting $p=0$ in (\ref{eqnxy})  we get
$$
y_i^{(0)} = x_{i+2}^{(q)} \, \theta^{-q},
$$
which, along with \eqref{eq1} for $p=0$, gives the following set of equations
\begin{equation}\label{eq2}
x_{i+2}^{(q)} = x_{i+2}^{(0)} \, \theta^q,
\end{equation}
where $q=1, \ldots, m-1$ and $i=1, \ldots, n-3$. Thus it follows from \eqref{eq1}, \eqref{eq2} and the above remark, that $\varphi$ depends on the following $2m+n-3$ parameters only:
$$
x_{i+2}^{(0)}, \quad i=1, \ldots, n-3, \qquad x_2^{(k)}, \quad y_{n-2}^{(k)}, \quad
k=0, \ldots, m-1.
$$
Hence
$$
| \Aut_c(G_n^{(m)}) | \leqslant p^{2m^2 + mn - 3m},
$$
since each of these parameters can have at most $p^m$ values in $\mathbb F_{p^m}$.

Now we proceed to show that $p^{2m^2 + mn - 3m} \le | \Aut_c(G_n^{(m)}) |$. Fix $k$ and $\ell$ belonging to $\{ 0, 1, \ldots, m-1\}$. Suppose
$$
\phi^{k,\ell} (t_{3,2} (\theta^k)) = (t_{3,2} (\theta^k))^{\displaystyle t_{2,1} (\theta^{\ell})} = t_{3,2} (\theta^k) \, [ t_{3,2} (\theta^k), t_{3,1} (\theta^{\ell}) ] = t_{3,2} (\theta^k) \, t_{3,1} (\theta^{k+\ell})
$$
and $\phi^{k,\ell} (t_{i+1,i}(\theta^j)) = t_{i+1,i}(\theta^j)$ if $(i,j) \neq (2,k)$. It is easy to see that $\phi^{k,\ell}$ is a basis conjugating automorphism. Notice that $\phi^{k,\ell}$ induces a basis conjugating automorphism on the subgroup consisting of $3 \times 3$ top left blocks of all matrices of $G_n^{(m)}$ and it fixes every element outside this subgroup. Hence, by Lemma \ref{lemma_n=3},  $\phi^{k,\ell}$ is class preserving. Let $A_1$ denote the subgroup of $\Aut(G_n^{(m)})$ generated by $\{\phi^{k, \ell} \mid 0 \le k, \ell \le m-1\}$. Thus we get $p^{m^2}$ class preserving automorphisms of $G_n^{(m)}$.

Again fix $k$ and $\ell$ belonging to $\{ 0, 1, \ldots, m-1\}$. Suppose
$$
\psi^{k,\ell} (t_{n-1,n-2} (\theta^k)) =  (t_{n-1,n-2} (\theta^k))^{\displaystyle t_{n,n-1} (\theta^{\ell})} =  t_{n-1,n-2} (\theta^k) \, t_{n,n-2} (- \theta^{k+\ell})
$$
and $\psi^{k,\ell} (t_{i+1,i}(\theta^j)) = t_{i+1,i}(\theta^j)$ if $(i,j) \neq (n-2,k)$. It is again easy to see that $\psi^{k,\ell}$ is basis conjugating automorphism. Notice that $\psi^{k,\ell}$ induces a basis conjugating automorphism on the subgroup consisting of $3 \times 3$ right bottom blocks of all matrices of $G_n^{(m)}$ and it fixes every element outside this subgroup. Hence, by Lemma \ref{lemma_n=3}, $\psi^{k,\ell}$ is class preserving. Let $A_2$ denote the subgroup of $\Aut(G_n^{(m)})$ generated by $\{\psi^{k, \ell} \mid 0 \le k, \ell \le m-1\}$. Thus we again get $p^{m^2}$ class preserving automorphisms of $G_n^{(m)}$. Notice that $A_1$ and $A_2$ intersect trivially.

Now, fix $j \in \{ 2, \ldots, n-2 \}$ and $\ell \in \{ 0, \ldots, m-1\}$. Then conjugations
\begin{eqnarray*}
f^{\ell}_j (t_{i+1, i} (\theta^k)) & = & (t_{i+1,i} (\theta^k))^{\displaystyle t_{j+1,j}\, (\theta^{\ell})}
\end{eqnarray*}
give $p^{(n-3)m}$ inner automorphisms.

Since the set of these inner automorphisms intersects trivially with $A_1$ and $A_2$, it follows that
$$
|\Aut_c(G_n^{(m)}) | \ge p^{2m^2 + m(n - 3)}.
$$

Hence $|\Aut_c(G_n^{(m)})| = p^{2m^2 + mn - 3m}$. Since $|\Inn(G_n^{(m)})| = p^{m(n-1)}$, we have
$$
|\Out_c(G_n^{(m)})| = p^{2m(m - 1)}.
$$
Since the exponent of $\gamma_2(G_n^{(m)})$ is $p$, it follows that $\Aut_c(G_n^{(m)})$ is an elementary abelian group, and therefore $\Out_c(G_n^{(m)})$ is also an elementary abelian group. This completes the proof of the theorem.
\end{proof}

\section{Open problems on normal automorphisms}

In this section we mainly talk about the group of normal automorphisms of $G$, which consists of automorphisms of $G$ fixing all normal subgroups of~it.

Normal automorphisms for various classes of groups were investigated intensively. We recall some of the results obtained in this direction.

By \cite{Lub, Lue} if $G$ is free non-abelian group, then  each of its normal automorphism is inner. Thus
$$
\Inn(G) = \Aut_c(G) = \Aut_n(G).
$$

Let $P_n$, $n \geqslant 3$, be the pure braid group. It was proved by Nesh\-cha\-dim \cite{Ne-2}, that any normal automorphism of $P_n$ is inner. It is well-known that $P_n$ is a semidirect product of certain free groups.

\begin{problem}
Let $G = F_k \leftthreetimes F_{\ell}$, where $k, \ell \geqslant 2$, be a semidirect product of free non-abelian groups. Is it true that every normal automorphism of $G$ is inner? If it is not true in general, then is it true in the case when $F_{\ell}$ acts on $F_k$ by identity modulo the derived subgroup?
\end{problem}

Romankov \cite{R} proved that if $G$ is a free  non-abelian solvable group, then $\Aut_n(G) = \Inn( G)$. Normal automorphisms of pro-finite groups were studied in~\cite{J}. Normal automorphisms of free two-step solvable pro-$p$-groups were investigated by Romanovskii and Boluts' in \cite{RB}. They described the group of normal automorphisms and proved that it is bigger than the group of inner automorphisms.  Romanovskii \cite{Rom} proved that  every normal automorphism of a free solvable pro-p-group of solvability step $\geq 3$ is inner.  It was shown by Jarden and Ritter \cite{JR} that all automorphisms of absolute Galois group $G(\mathbb{Q})$ are normal and all normal automorphisms are inner.

\begin{problem}
Describe the quotient $\Aut_n(G) / \Inn(G)$ for $G$ such that \linebreak 
$\Aut_n(G) \neq \Inn(G)$, and the quotient $\Aut_c(G) / \Inn(G)$ for
$G$ such that $\Aut_c(G) \neq \Inn(G)$.
\end{problem}
The following facts are known \cite{F-G}. If $G$ is nilpotent, then $\Aut_n(G)$ is nilpotent-by-abelian. If $G$ is polycyclic, then $\Aut_n(G)$ is polycyclic. In particular, if $G$ is finite solvable, then so is $\Aut_n(G)$.

\begin{problem}[Endimioni \cite{En1}]
Does the last statement hold true without finiteness hypothesis?
\end{problem}

Endimioni \cite{En1} proved that the group of all normal automorphisms of a metabelian group is solvable of length $\leqslant 3$. Moreover, he gave an example showing that this statement can not be improved. Indeed,  $\Aut_n(A_4) = \Aut(A_4) = S_4$, where $A_4$ is metabelian and $S_4$ is solvable of length $3$.

It was proven by Robinson \cite{Ro} that for any finite group $F$, there exists a finite semi-simple group $G$ such that $\Aut_n(G) /  \Inn(G)$ contains $F$ as a subgroup.

\begin{problem}
Does the last statement hold true without finiteness hypothesis?
\end{problem}

The following example shows that there exists an abelian group $G$ such that $\Aut_n(G) /$ $ \Aut_c(G) \neq 1$.

\begin{example}
{\rm Let $G = C = \langle a \rangle$ be an infinite cyclic group. Obviously, $\Inn(C) = \Aut_c(C)  = 1$, and automorphism $\varphi : a \longmapsto a^{-1}$ belongs to $\Aut_n(C)$. Each subgroup of $C$ is normal and $\varphi$ sends each of its subgroup to itself. Therefore, $\Aut_n(C) / \Aut_c(C) = \Aut(C) \simeq C_2$ is the cyclic group of order two. }
\end{example}

The following example shows that there exists a non-abelian group $G$ such that $\Aut_n(G) /$ $ \Aut_c(G) \neq 1$.

\begin{example} \label{example1.2} {\rm Consider $G  = \langle x, y ~|~ x^2 = y^2 \rangle$. It was pointed out by Neshchadim \cite{Ne-2} that the automorphism
$$
\varphi :
\left\{%
\begin{array}{l}
x \longmapsto x^{-1}, \\
y \longmapsto y^{-1},\\
\end{array}%
\right.
$$
is normal, but not inner. Let us show that this  automorphism is not class preserving. Indeed, $\varphi$ send central element $x^2$ to $x^{-2}$. Obviously, $x^{-2}$ and $x^2$ are not conjugate, because $g^{-1} x^2 g = x^2$ for any $g \in G$. Therefore, $\Aut_n(G) / \Aut_c(G) \neq 1$. }
\end{example}

For $p, q \in \mathbb{Z}$, define the group
$$
G(p, q) = \langle x, y ~|~ x^p = y^q \rangle .
$$
If $p>0$ is odd and $0<q<p$ with $(p,q)=1$, then $G(p,q)$ is a
fundamental group of $(p,q)$-torus knot complement in the
$3$-sphere. Moreover, $G(3,2) \simeq B_3$, where $B_n$ is the braid
group on $n$ strands.  Neshchadim \cite{Ne-1} proved that
 every normal automorphism  of $B_n$ is inner. As we mentioned in
Example~\ref{example1.2}, $\Aut_n(G(2, 2)) \neq \Inn(G(2, 2))$ and
$\Aut_n(G(2, 2)) \neq \Aut_c(G(2, 2))$.

\begin{problem}
Describe $p$ and $q$ such that $\Aut_c( G(p, q)) \neq \Inn(G(p,
q))$. Analogously, describe $p$ and $q$ such that $\Aut_n(G(p, q))
\neq \Inn(G(p, q))$.
\end{problem}

The following problem is regarding knot groups.

\begin{problem}
Do normal (or class preserving) automorphisms of a knot group are all inner? This is true for the trefoil knot group $G(3,2)$.
\end{problem}

It is known that for free nilpotent group of class $2$, every automorphism is tame (an automorphism is said to be \emph{tame} if it is induced by an automorphism of a free group under the natural homomorphism). It was proved by Neshchadim \cite{N1} that every tame normal automorphism of $F_n / \gamma_4 (F_n)$, where $n \geqslant 2$, is inner. Moreover, he gave a complete description, in terms of generators and relations, of the group of normal automorphisms of a free nilpotent group of class $4$ with arbitrary number of generators.

\begin{problem}
What can one say about tame normal automorphisms of \linebreak  $F_n / \gamma_m(F_n)$ if $n \geqslant 2$ and $m \geqslant 5$?
\end{problem}

\appendix

\section{Alternative proof of Theorem \ref{theorem-new}}

The following proof of our Theorem \ref{theorem-new} is suggested by
N. S.~Romanovskii.

\noindent{\bf Alternative proof of Theorem \ref{theorem-new}.}  First we show that all class preserving automorphisms of $\Gamma_{n, n} \cong \mathrm{UT}_n(K)$ are inner.
Since the group  $\mathrm{UT}_2(K)$ is abelian,  therefore $\Aut_c(\mathrm{UT}_2(K)) = \Inn(\mathrm{UT}_2(K)) = 1$.  By inductive argument, suppose that $\Aut_c(\mathrm{UT}_{n-1}$ $(K)) = \Inn(\mathrm{UT}_{n-1}(K)) $ for some integer $n \ge 3$. Let $G = \mathrm{UT}_n(K)$. Define two subgroups $A$ and $B$ of $G$ as follows.
$$A = \{(a_{ij}) \in G \mid a_{ij} = 0 ~\mbox{for all}~~i \le n-1\},$$
and
$$B = \{(b_{ij}) \in G \mid b_{nj} = 0 \}.$$
Notice that $A$ is  an abelian normal subgroup of $G$ generated by subgroups $t_{n,j} (K)$, where $1 \leq j \leq n - 1$, and $B$ is generated by subgroups $t_{i+1,i} (K)$, where $1 \leq i \leq n-2$. The group $G$  is a semidirect product of $A$ and $B$. Thus $B \cong G/A$ and every $g \in G$ can be presented as $g = b a$, where $b \in B$, $a \in A$.

Let $\varphi \in \Aut_c(G)$.  Consider the action of $\varphi$ on the generators of $B$:
$$
\varphi (t_{i, i-1}) = t_{i,i-1}^{b_{i-1} a_{i-1}}, \quad i=2, \ldots, n-1
$$
where $b_{i-1} \in B$, $a_{i-1} \in A$, and $t_{i,i-1} = t_{i,i-1} (1)$. Then $\varphi$ induces a class preserving automorphism $\overline{\varphi}$ of $B$, where
$$
\overline{\varphi} (t_{i, i-1}) = t_{i,i-1}^{b_{i-1}}, \quad i=2, \ldots, n-1.
$$
Since $G/A \cong B \cong \mathrm{UT}_{n-1}(K)$ and  $\Aut_c(\mathrm{UT}_{n-1}(K)) = \Inn(\mathrm{UT}_{n-1}(K)) $, it follows that $\overline{\varphi} \in \Inn(B)$. Then there exists an element $b \in B$ such that $\overline{\varphi} $ is the inner automorphism of $B$  induced by $b$ under conjugation. Denote by $f_{b}$ the inner automorphism of $G$  induced by $b\in B$ under conjugation. Notice that $f_b^{-1} = f_{b^{-1}}$. Consider the action of  $f_b^{-1}\varphi$ on the generators of the group $B$:
$$
f^{-1}_{b} \varphi (t_{i+1,i}) = f^{-1}_{b} (\varphi (t_{i+1,i})) = \left( t_{i+1, i}^{b_{i} a_{i}} \right)^{b^{-1}} = t_{i+1,i}^{b b^{-1} a_{i } [ a_{i}, b ] } = t_{i+1,i}^{{\bar{a}_{i}}} ,
$$
where $\bar{a}_{i}   = a_{i}  [a_{i}, b_{i}]$. Let us represent
$$
\bar{a}_{i} = t_{n,1} (\alpha_{i,1}) t_{n,2} (\alpha_{i,2}) \ldots t_{n,n-1} (\alpha_{i,n-1}) ,
$$
for some $\alpha_{i,j} \in K$. Using the commutator identity, one can easy check that
$$
t_{t+1,i}^{\bar{a}_{i}} = t_{i+1, i}^{ t_{n,1} (\alpha_{i,1}) t_{n,2} (\alpha_{i,2}) \ldots t_{n,n-1} (\alpha_{i,n-1})  }
 = t_{i+1,i}^{t_{n,i+1}(\alpha_{i,i+1})} = t_{i+1,i} \, t_{n,i}(-\alpha_{i,i+1}).
$$
Consider the following element of $A$:
$$
a =  t_{n,2} (\alpha_{1,2}) t_{n,3} (\alpha_{2,3}) \ldots t_{n,n-1} (\alpha_{n-2,n-1}) .
$$
Then for every $i=1,2, \ldots, n-2$ the following equality holds:
$$
t_{i+1,i}^{\bar{a}_{i}} = t_{i+1,i} ^{a} .
$$
Hence the automorphism $\psi = f_{a}^{-1} f_{b}^{-1} \varphi$ of $G$ fixes all generators of $B$.

Consider the action of $\psi$ on $t_{n,n-1} \in A$. Since $\psi$  is class preserving, it acts as follows:
$$
\psi (t_{n,n-1})  = t_{n,n-1}^{b_{n-1} a_{n-1}} ,  \qquad b_{n-1} \in B, a_{n-1} \in A .
$$
Since $A$ is an abelian normal subgroup of $G$, it follows that  $t_{n,n-1}^{b_{n-1} a_{n-1}}  =  t_{n,n-1}^{b_{n-1}} $. Presenting
$$
t_{n,n-1}^{b_{n-1}} = t_{n,n-1} [t_{n,n-1}, b_{n-1}]
$$
and using the fact $[t_{n,n-1}, b_{n-1}] \in A$, we can write
$$
t_{n,n-1}^{b_{n-1}} = t_{n,n-1} t_{n,1} (\beta_{1}) t_{n,2} (\beta_{2})  \ldots t_{n,n-2} (\beta_{n-2}) .
$$
It follows that $\beta_{2}  = \beta_{3}  = \ldots = \beta_{n-2} = 0$. Indeed, for any $2 \le k \le n-2$, we have
 \begin{eqnarray*}
1 &=& \psi([t_{n, n-1}, t_{k,k-1}]) = [  t_{n,n-1} t_{n,1} (\beta_{1}) t_{n,2} (\beta_{2}) \ldots t_{n,n-2} (\beta_{n-2}), t_{k,k-1}] \cr
& = & [t_{n,n-1} t_{n,1} (\beta_{1})  \ldots t_{n,k-1}(\beta_{k-1}) t_{n,k+1}(\beta_{k+1}) \ldots t_{n,n-2} (\beta_{n-2}), t_{k,k-1}]^{t_{n,k}(\beta_{k})}\cr
& & [t_{n,k}(\beta_{k}), t_{k,k-1}] \cr
& = &  t_{n,k-1} (\beta_{k}).
\end{eqnarray*}
Hence $\beta_{2} = 0$. Thus
$$
\psi(t_{n,n-1})  = t_{n,n-1} t_{n,1} (\beta_{1})  = t_{n,n-1}^{c},
$$
where $c = t_{n-1,1} (\beta_{1})$. Therefore, $f_{c}^{-1} \psi $ fixes $ t_{n,n-1}$.  Moreover, $c$ commutes with all generators of $B$,  since $c \in Z(B)$. Hence  $f_{c}^{-1} \psi = id_{G}$. Thus, we proved that $\varphi$ is a product of inner automorphisms of $G$. So $\varphi$ itself is an inner automorphism of $G$ and therefore  $\Aut_c(\Gamma_{n, n}) = \Inn(\Gamma_{n, n})$.

Now we proceed to prove the general case, i.e., $\Aut_c(\Gamma_{n, k}) = \Inn(\Gamma_{n, k})$, $1 \le k \le n$. Since $\Aut_c(\Gamma_{n ,n}) = \Inn(\Gamma_{n, n})$ for all $n$ (as we proved above) and $\Gamma_{3, k}$ is abelian for $1 \le k \le 3$, by inductive argument we can assume that $\Aut_c(\Gamma_{m, k}) = \Inn(\Gamma_{m, k})$ for all $m < n$ and $1 \le k \le m$. So assume that $n \ge 4$ and $3 \le k < n$. Let $A'$ denote the subgroup of $\Gamma_{n, k}$ generated by the subgroups $t_{n, j}(K)$, where $n-k+1 \le j \le n-1$ and $B'$ denote the subgroup generated by the subgroups $t_{i+1, i}(K)$, where $1 \le i \le n-2$. Here we denoted by $t_{i,j} (\alpha)$ the image of element $t_{i,j} (\alpha) \in G$ under the natural homomorphism $G \to \Gamma_{n,k}$.  Notice that  $A'$ is an abelian normal subgroup of $\Gamma_{n, k}$ and $B' \cong \Gamma_{n, k}/A' \cong \Gamma_{n-1, k}$.

Let $\varphi \in \Aut_c(\Gamma_{n, k})$. Then using arguments similar to the case $\Gamma_{n, n}$, we can find an automorphism $\psi$ of $\Gamma_{n, k}$ which fixes each element of $B'$. After that we can prove that $\psi$ is an inner automorphism of $\Gamma_{n,k}$. Hence $\varphi$ is an inner automorphism.

\end{document}